\documentclass{amsart}
\usepackage{amsmath,amssymb,amsthm}
\usepackage[all]{xy}
\newtheorem{dfn}{Definition}[section]
\newtheorem{thm}[dfn]{Theorem}
\newtheorem{lem}[dfn]{Lemma}
\newtheorem{cor}[dfn]{Corollary}
\newtheorem{rem}[dfn]{Remark}
\newtheorem{prop}[dfn]{Proposition}
\newtheorem{ex}[dfn]{Example}
\newtheorem{conj}[dfn]{Conjecture}

\newcommand{\Aut}{\mathrm{Aut}}

\newcommand{\rk}{\mathrm{rk}}
\newcommand{\Hodge}{H^{2,2}(X,\mathbb{Z})}
\newcommand{\pr}{\mathrm{pr}}
\newcommand{\hkmfd}{hyperK\"ahler manifold}
\newcommand{\hkmfds}{hyperK\"ahler manifolds}
\newcommand{\htop}{h_{\mathrm{top}}}
\newcommand{\hcat}{h_{\mathrm{cat}}}
\newcommand{\ZZ}{\mathbb{Z}}

\newcommand{\D}{\mathcal{D}}

\newcommand{\stab}{\mathrm{Stab}^*}
\newcommand{\Halg}{ \widetilde{H}^{1,1}(S,\mathbb{Z})}
\newcommand{\A}{\mathcal{A}}
\newcommand{\C}{\mathcal{C}}
\newcommand{\KA}{K_{\mathrm{num}}(\A_X)}
\newcommand{\OH}{O_{\mathrm{Hodge}}}
\newcommand{\OHH}{O^{+}_{\mathrm{Hodge}}}
\begin{document}
\title[AUTOMORPHISMS OF POSITIVE ENTROPY ON SOME HYPERK\"AHLER MANIFOLDS]{AUTOMORPHISMS OF POSITIVE ENTROPY ON SOME HYPERK\"AHLER MANIFOLDS VIA DERIVED AUTOMORPHISMS OF K3 SURFACES}
\author{GENKI OUCHI}
\maketitle
\begin{abstract}
We construct examples of hyperK\"ahler manifolds of Picard number two with automorphisms of positive entropy via derived automorphisms of K3 surfaces of Picard number one. Our hyperK\"ahler manifolds are constructed as moduli spaces of Bridgeland stable objects in derived categories of K3 surfaces. Then automorphisms of positive entropy are induced by derived automorphisms of positive entropy on K3 surfaces.
\end{abstract}
\section{INTRODUCTION} 
\subsection{Motivation and Results}
The dynamical aspect of automorphisms of K3 surfaces was studied by Cantat \cite{C} and McMullen \cite{Mc}. It gave a new perspective of the study of automorphisms of algebraic varieties. After that, Oguiso studied  the case of \hkmfds. A \hkmfd \ is a simply connected compact K\"ahler manifold with an everywhere non-degenerate holomorphic two form unique up to scalar.  Typical examples are K3 surfaces and Hilbert schemes of points on them. More generally, Mukai \cite{Muk84} proved that moduli spaces of Gieseker stable sheaves on projective K3 surfaces give examples of   hyperK\"ahler manifolds deformation equivalent to Hilbert schemes of points on K3 surfaces i.e. of $\text{K3}^{[n]}$-type. In this paper, we study the dynamical aspect of automorphisms of moduli spaces of Gieseker stable sheaves on projective K3 surfaces using technique of derived categories. A pair $(X,f)$ of a compact Hausdorff space $X$ and a homeomorphism $f: X \to X$ is called  a discrete topological dynamical system. Then we are interested in behavior of iteration of the homeomorphism $f$.  The topological entropy $\htop(f)$ of $f$ is a fundamental invariant of the discrete topological dynamical system $(X,f)$, which measures complexity of $(X,f)$. The positivity of the topological entropy $\htop(f)$ means that $(X,f)$ is complicated. For example, a homeomorphism of finite order gives the null entropy. From now on, we focus on dynamics on projective \hkmfds. It is known that birational automorphism groups of  \hkmfds \ of Picard number one are finite groups.

\begin{lem}$($\cite{Huy99}$)$\label{kernel}
Let $M$ be a hyperK\"ahler manifold. Consider a homomorphism $\mathrm{Bir}(M) \to \mathrm{O}(H^2(M,\mathbb{Z})), f^{-1} \mapsto f^*.$ Then it's kernel is a finite group.
\end{lem}
\begin{cor}\label{picard}
Let $M$ be a  hyperK\"ahler manifold of Picard number one. Then $\mathrm{Bir}(M)$ is a finite group.
\end{cor}

Since automorphisms of positive entropy have infinite order, \hkmfds \ of Picard number one do not have automorphisms of positive entropy. In the case of Picard number two, Oguiso found a two dimensional example \cite{Ogu2} and a four dimensional example \cite{Ogu1} of \hkmfds \ with automorphisms of positive entropy.  They are constructed via the study of period maps. The geometry of the two dimensional example can be understood explicitly \cite{FGGL}. On the other hand, there are no explicit higher dimensional examples of \hkmfds \ of Picard number two with automorphisms of positive entropy so far. In general setting, Amerik and Verbitsky proved that the following theorem recently.
\begin{thm}$($\cite{AV16}$)$\label{AV}
Let $M$ be a hyperK\"ahler manifold with $b_2(M)\geq 5$. Then $M$ admits a projective deformation $M^{\prime}$ of Picard number two with symplectic automorphisms of positive entropy. 
\end{thm}

 The main theorem of this paper gives explicit construction of such \hkmfds \ as moduli spaces of Gieseker stable sheaves on projective K3 surfaces of Picard number one. Moreover, we discuss relations between automorphisms of moduli spaces and autoequivalences of derived categories of projective K3 surfaces. We will give the precise statement in Theorem \ref{mainintro}. In the case of Picard number three, we can easily construct higher dimensional examples of \hkmfds \ with automorphisms of positive entropy. Let $S$ be the Oguiso's K3 surface of Picard number two with an automorphism $f$ of positive entropy. Then the natural automorphism $f^{[n]}: S^{[n]} \to S^{[n]}$ of $f$ has a positive  entropy for any positive integer $n$. It is well known that automorphism groups of projective K3 surfaces of Picard number one are $\ZZ/2\ZZ$ or $1$. The first case occur only for projective K3 surfaces of degree two. So natural automorphisms are not interesting in this case. Moreover, the following holds.
 
 \begin{prop}$($\cite{Ogu1}$)$\label{hilb}
Let $S$ be a projective K3 surface of Picard number one. Let $n>0$ be a positive integer.
Then the birational automorphism group $\mathrm{Bir}(\mathrm{Hilb}^n(S))$ of the Hilbert scheme of points on $S$ is a finite group.
\end{prop}

The main idea of this paper is that we use autoequivalences of derived categories of K3 surfaces instead of automorphisms on them. Dimitrov, Haiden, Katzarkov and Kontsevich \cite{DHKK} introduced the notion of the categorical entropy $\hcat(\Phi)$ of an endofunctor $\Phi$ on a triangulated category $\D$.   In the context of the algebraic geometry, this is a generalization of topological entropy in the sense of the following Theorem due to Kikuta and Takahashi \cite{KT16}.
\begin{thm}$($\cite{KT16}$)$
Let $X$ be a smooth projective variety and $f:X\to X$ be a surjetive endomorphism. Then we have 
\[ h_{\mathrm{cat}}(\mathbf{L}f^*)=h_{\mathrm{top}}(f). \]
\end{thm}

The first observation of this paper is the following proposition. 
\begin{prop}\label{positivederived}
Let $S$ be a projective K3 surface. Then there is an autoequivalence $\Phi \in \Aut(D^b(S))$ such that $\hcat(\Phi)>0$.
\end{prop}
Even for a projective K3 surface of Picard number one, we have autoequivalences of positive entropy. We hope that some autoequivalences of positive categorical entropy as in Proposition \ref{positivederived} can be  understood as symmetry of moduli spaces of Bridgeland stable objects. Bayer and Macri \cite{BM12}, \cite{BM13} proved that moduli spaces of Bridgeland stable objects in derived categories of projective K3 surfaces are projective \hkmfds \ of $\text{K3}^{[n]}$-type.  By Bayer and Macri's work, moduli spaces of Gieseker stable sheaves on projective K3 surfaces can be described as moduli spaces of stable objects with respect to some Bridgeland stability condition. Since autoequivalences change stability conditions, autoequivalences do not induce automorphisms of moduli spaces of stable objects in general. So we study variation of stability conditions, namely wall and chamber structures on the space of stability conditions. For a given Mukai vector $v$, there is the wall and chamber structure in the space of stability conditions. Bayer and Macri proved that there are three types of walls, flopping walls, divisorial walls and fake walls. Before giving the main theorem, we fix some notation. Let $S$ be a projective K3 surface. Let $H^*(S,\ZZ)$ be the Mukai lattice and $\Halg$ be the $(1,1)$-part of the Mukai lattice. We denote the distinguished connected component of the space of stability conditions on $D^b(S)$ by $\stab(S)$.  For a primitive vector $v \in \Halg$ and a $v$-generic stability condition $\sigma \in \stab(S)$, the moduli space $M_\sigma(v)$ of $\sigma$-stable objects with Mukai vector $v$ is a projective \hkmfd \ of $\text{K3}^{[n]}$-type and $2n=v^2+2$ \cite{BM12}, \cite{BM13}.  The main results of this paper are as follows.

\begin{thm}\label{mainintro}$($\rm{Theorem \ref{converse} Corollary \ref{maincor}}$)$\it{}
Let $S$ be a projective K3 surface of Picard number one. Let $v \in \Halg$ be a primitive vector with $v^2 \geq2$. Assume that all walls with respect to $v$ are fake walls. For any $v$-generic stability condition $\sigma \in \stab(S)$, there is an autoequivalence $\Phi \in \Aut(D^b(S))$ such that the morphism
\[ \phi : M_\sigma(v) \to M_\sigma(v), \ [E] \mapsto [\Phi(E)] \]
is an automorphism of $M_\sigma(v)$ of positive entropy $\htop(\phi)>0$. Moreover, we have an inequality
\[ \frac{1}{2}\dim M_\sigma(v) \cdot \hcat(\Phi) \geq \htop(\phi).\]
\end{thm}

If Conjecture \ref{KTconj} (Conjecture 5.3 in \cite{KT16}) is true, then the following holds.
\begin{conj}In the setting in Theorem \ref{mainintro}, we have the equality
 \[ \frac{1}{2}\dim M_\sigma(v) \cdot \hcat(\Phi) = \htop(\phi).\]
\end{conj}

We can construct examples of Mukai vectors having only fake walls.
\begin{ex}\label{mainexintro}
Let $S$ be a K3 surface with $\mathrm{NS}(S)=\ZZ h$. Assume one of the followings.

\begin{itemize}
\item[\rm{(1)}]\it{}$h^2=132, v=(4,h,16), v^2=4$
\item[\rm{(2)}]\it{}$h^2=510, v=(6,h,42), v^2=6$
\item[\rm{(3)}]\it{}$h^2=1160, v=(8,h,72), v^2=8$
\item[\rm{(4)}]\it{}$h^2=2210, v=(10,h,110), v^2=10$
\end{itemize}
Then there are only fake walls with respect to $v$.
\end{ex}

We will construct four dimensional examples of Theorem \ref{mainintro} from cubic fourfolds.

\subsection{From cubic fourfolds}
The geometry of cubic fourfolds give examples of \hkmfds. Beauville and Donagi \cite{BD} proved that the Fano scheme $F(X)$ of lines on a cubic fourfold $X$ is a hyperK\"ahler fourfold of $\text{K3}^{[2]}$-type. CF.Lehn, M. Lehn, CH. Sorger and  D. Van straten \cite{LLSS} constructed the  hyperK\"ahler eightfold $Z(X)$ of $\text{K3}^{[4]}$-type from twisted cubic curves on a cubic fourfold $X$ not containing a plane.  Conjecturely, these hyperK\"ahler manifolds can be described as moduli spaces of stable objects in some K3 categories constructed from cubic fourfolds. We recall the construction of the K3 category. Let $X$ be a cubic fourfold. There is a semi-orthogonal decomposition 
 \[ D^b(X)=\langle \A_X, \mathcal{O}_X, \mathcal{O}_X(1), \mathcal{O}_X(2) \rangle. \]
Kuznetsov \cite{Kuz03}, \cite{Kuz10} proved that $\A_X$ is a K3 category i.e. $[2]$ is a Serre functor. Let $\pr : D^b(X) \to \A_X$ be the left adjoint of the inclusion functor. We expect that $F(X)$ is a moduli space of stable objects in $\A_X$ with the numerical class $[\pr(\mathcal{O}_{\mathrm{line}}(1))]$ \cite{KM}, \cite{MS}. Similarly, the hyperK\"ahler eightfold $Z(X)$ is constructed as a moduli space of stable objects in $\A_X$ with the numerical class  $[\pr(\mathcal{O}_{\mathrm{point}}(1))]$ \cite{AT}, \cite{AL}, \cite{O} conjecturelly. In this paper, we treat only Fano schemes of lines on cubic fourfolds. Sometimes, the K3 category $\A_X$ is equivalent to the derived category of some K3 surface. Kuznetsov proposed the following conjecture. 

   \begin{conj}$($\cite{Kuz10}$)$\label{Kuzconj}
   A cubic fourfold $X$ is rational if and only if there is a K3 surface $S$ such that $\mathcal{A}_X$ is equivalent to $D^b(S)$.
   \end{conj}
   Note that there are no known irrational cubic fourfolds so far. We recall works by Hassett \cite{Has00}, Addington and Thomas \cite{AT}, Galkin and Shinder \cite{GS} and Addington \cite{Ad}. Let $\mathcal{C}$ be the moduli space of cubic fourfolds. Let $X$ be a  cubic fourfold. We say that $X$ is special if $\mathrm{rk}H^{2,2}(X,\mathbb{Z}) \geq2$ holds.  Otherwise, $X$ is called very general. More specifically, if $X$ has a rank $2$ primitive sublattice $K \subset H^{2,2}(X,\mathbb{Z})$ of discriminant $d$ such that $H^2 \in K$, then $X$ is called a special cubic fourfold of discriminant $d$.  Here, $H \in H^2(X,\mathbb{Z})$ is the hyperplane class.  Let $\mathcal{C}_d$ be the (possibly empty) codimension one subvariety of special cubic fourfolds of discriminant $d$.   Hassett proved that $\mathcal{C}_d$ is not empty if and only if 
 \begin{center}($*$): $d>6$ and $d \equiv 0$ or $2$ (mod $6$).\end{center} 
 Moreover, Hassett \cite{Has00} proved that  $X$ has an associated K3 surface at the level of Hodge theory if and only if $d$ satisfies ($*$) and 
  \begin{center}($**$): $d$ is not divisible by $4$, $9$ or any odd prime $p\equiv2$ (mod $3$). \end{center}  
  
  Addington and Thomas \cite{AT} proved the following theorem from the viewpoint of derived categories.
 \begin{thm}$($\cite{AT}$)$
   Let $X$ be a cubic fourfold. 
   \begin{itemize}
   \item If there is a K3 surface $S$ such that $\mathcal{A}_X \simeq D^b(S)$, then we have $X \in \mathcal{C}_d$ for some integer $d$ satisfying $(*)$ and $(**)$.
   \item Let $d$ be an integer satisfying $(*)$ and $(**)$. Then the set 
   \[U_d:=\{X \in \mathcal{C}_d \mid \mathcal{A}_X \simeq D^b(S) \} \subset \mathcal{C}_d \]
   is a non-empty Zariski open subset.
   \end{itemize}
   \end{thm}
   
   By Galkin and Shinder's work \cite{GS}, we have the following conjecture. 
   
\begin{conj}\label{GSconj}
If $X$ is a rational cubic fourfold, then $F(X)$ is birational to the Hilbert scheme of two points on some K3 surface.
\end{conj}   
This conjecture comes from calculation in the Grothendieck ring of algebraic varieties. They proved Conjecture \ref{GSconj} under some assumption.  However, there are counterexamples of their assumptions (See \cite{B}, \cite{M}, \cite{IMOU}). Addington \cite{Ad} proved that Conjecture \ref{Kuzconj} and Conjecture \ref{GSconj} are not compatible. Conjecture \ref{GSconj} is true for known rational cubic fourfolds\cite{BD}, \cite{Ad}, \cite{Op}.

\begin{thm}$($\cite{Ad}$)$\label{Adthm}
 Let $X$ be a cubic fourfold.
 \begin{itemize}
 \item There is a K3 surface $S$ such that $F(X)$ is birational to a moduli space of  stable sheaves on $S$  if and only if  $d$ satisfies $(*)$ and 
  $(**)$.  
  \item There is a K3 surface $S$ such that $F(X)$ is birational to a Hilbert scheme of two points on $S$ if and only if there is an integer $d$ satisfying $(*)$ and
  \begin{center}
 $(***)$: The equation $a^2d=2n^2+2n+2$ has an integral solution $(a,n)$.
\end{center} 
 \end{itemize}
 The integer $74$ satisfies $(*)$ and $(**)$. However, it does not satisfy  $(***)$.
 \end{thm}
 
We study automorphisms of Fano schemes of lines on special cubic fourfolds of discriminant $74$.  

\begin{thm}\label{cubicintro}
Take $X \in \C_{74}$ such that $\rk\Hodge=2$. Then there is an automorphism $\phi \in \Aut(F(X))$ such that $\htop(\phi)>0$.
\end{thm}

Theorem \ref{cubicintro} gives examples of Theorem \ref{mainintro}.

\subsection{Plan of the paper}
In Section 2, we review the lattice theory,  hyperK\"ahler manifolds, Bridgeland stability conditions on K3 surfaces, moduli spaces of stable objects,  the theory of entropy and some results about automorphisms of  hyperK\"ahler manifolds.
In Section 3, we prove Proposition \ref{positivederived}, Theorem \ref{mainintro}, Example \ref{mainexintro} and Theorem \ref{converse}.
In Section 4, we prove Theorem \ref{cubicintro}.

\subsection{Notation}
We work over the complex number field $\mathbb{C}$. For a triangulated category $\D$ of finite type over $\mathbb{C}$, let $K(\D)$ be the Grothendieck group of $\D$. For objects $E,F \in \D$, we define $\chi(E,F):=\sum_{i\in\ZZ}\mathrm{dim}\mathrm{Ext}^i(E,F)$. we define the numerical Grothendieck group $K_{\mathrm{num}}(\D)$ as the quotient group $K(\D)/\mathrm{Ker}\chi$. Here, the subgroup $\mathrm{Ker}\chi$ consists of objects $E$ such that $\chi(E,F)=0$ for any object $F \in \D$.
If $\D$ is the derived category of a smooth projective variety $X$, we write $K(X)$ (resp. $K_{\mathrm{num}}(X)$) shortly.  For a smooth projective variety $X$, we denote the real vector space of $1$-cycles on $X$ modulo numerical equivalence by $N_1(X)$.  We assume that hyperK\"ahler manifolds are projective. We write the spherical twist $\mathrm{ST}_E$ of a spherical object $E \in D^b(X)$ for a smooth projective variety. 

 \subsection*{Acknowledgements}
The author would like to express his sincere gratitude to Professor Yukinobu Toda and Professor Keiji Oguiso for valuable comments and warmful encouragement. The author also would like to thanks Naoki Koseki for listening his talk patiently.  This work was supported by the program for Leading Graduate Schools, MEXT, Japan.  This work is also supported by Grant-in-Aid for JSPS Research Fellow 15J08505. 

\section{PRELIMINARY}
\subsection{Lattices}
In this subsection, we collect fundamental facts about the lattice theory.
\begin{dfn}
A lattice is a pair $(L,q)$ of a finitely generated free abelian group $L$ and an integer-valued non-degenerate quadratic form $q$. A lattice $(L,q)$ is called even if $q(v) \in 2\mathbb{Z}$ holds for any $v\in L$. We often denote a lattice $(L,q)$ by $L$. The signature of $L$ is the signature of the quadratic form $q$.
\end{dfn}
Primitivity of sublattices and discriminants of lattices are used in the definition of special cubic fourfolds of discriminant $d$.
\begin{dfn}
Let $L$ be a lattice.  A sublattice $N$ of $L$ is called primitive if the quotient $L/N$ is torsion free. A vector $v \in L$ is called primitive if the sublattice $\mathbb{Z}v$ is primitive, equivalently  $v=\lambda w$ for some $\lambda \in \mathbb{Z}$ and $w \in L$ implies $\lambda = \pm1$.
\end{dfn}

\begin{dfn}\label{disc}
Let $L$ be a lattice.  The discriminant $\mathrm{disc}L$ of $L$ is the determinant of a Gram matrix of $L$. A lattice $L$ is called unimodular if the discriminant $\mathrm{disc}L$  is $\pm1$.
\end{dfn}
Recall the notion of discriminant groups to give the statement of Lemma \ref{Nik}.
\begin{dfn}
Let $(L,q)$ be an even lattice.  Using the quadratic form $q$, we get the natural injective map $L \to L^*$. The quotient group $L^*/L$ is called the discriminant group $d(L)$ of $L$. This is an finite abelian group. The discriminant form $q_L:d(L) \times d(L) \to \mathbb{Q}/2\mathbb{Z}$ is the quadratic form on $d(L)$ induced by the quadratic form $q$.
\end{dfn}

\begin{dfn}
Let $L$ be an even lattice. We denote the group of isometries of $L$ by $O(L)$. The group $O(L)$ is called the orthogonal group of $L$. We also denote the group of isometries of the discriminant group by $O(d(L))$. For $g \in O(L)$, $\overline{g}$ is the isometry of $d(L)$ induced by $g$. There is the  homomorphism $O(L) \to O(d(L))$ sending $g \to \overline{g}$.
\end{dfn}

The following lemma will be used in Proposition \ref{infiniteorder}.
\begin{lem}\label{Nik}$($\cite{N}$)$
Let $N$ be a primitive sublattice of an even unimodular lattice $(L,q)$. Consider an isometry $g \in O(N)$. If $\overline{g}=1$ holds, then there is a lift $\tilde{g} \in O(L)$ such that $\tilde{g}|_N=g$ and $\tilde{g}|_{L^{\perp}}=1$.
\end{lem}

From K3 surfaces, we get the notion of Mukai lattice.
\begin{dfn}
Let $S$ be a K3 surface. We define the Mukai pairing $\langle -, -  \rangle$ on $H^*(S,\mathbb{Z})$ as follow:
\[ \langle(r_1,c_1,m_1),(r_2,c_2,m_2) \rangle := c_1c_2-r_1m_2-r_2m_1.\]
The lattice $H^*(S,\mathbb{Z})$ is called the Mukai lattice and it is the even unimodular lattice of signature $(4,20)$.
The Mukai lattice $H^*(S,\mathbb{Z})$ has the weight two Hodge structure $\widetilde{H}(S)$ defined by $\widetilde{H}^{2,0}(S)=H^{2,0}(S)$. We denote the algebraic part of the Mukai lattice $H^*(S,\mathbb{Z})$ by 
\[ \widetilde{H}^{1,1}(S,\mathbb{Z}) = \Big(\bigoplus_{i=0}^{2}H^{i,i}(S,\mathbb{Q})\Big) \cap H^*(S,\mathbb{Z}). \]
We denote the group of Hodge isometries on $H^*(S,\mathbb{Z})$ by $\OH(H^*(S,\ZZ))$. Let $\OHH(H^*(S,\ZZ)) \subset \OH(H^*(S,\ZZ))$ be the subgroup of all Hodge isometries preserving positive definite 4-spaces in $H^*(S,\mathbb{R})$.  
\end{dfn}
We will use the derived Torelli theorem to construct autoequivalences.
Taking cohomological Fourier-Mukai transforms, we get the homomorphism $(-)^H :\mathrm{Aut}(D^b(S)) \to \OH(H^*(S,\ZZ))$. 
\begin{thm}\label{torelli}$($\cite{Muk87}, \cite{Orl97}, \cite{HMS09}$)$
The image $\mathrm{Im}(-)^H$ of the homomorphism $(-)^H$ is equal to $\OHH(H^*(S,\ZZ))$.
\end{thm}

\begin{dfn}
We denote the subgroup $\mathrm{Aut}^0(D^b(S))$ of $\mathrm{Aut}(D^b(S))$ as follow:
\[ \mathrm{Aut}^0(D^b(S)):= \mathrm{Ker}((-)^H : \mathrm{Aut}(D^b(S)) \to\OHH(H^*(S,\ZZ)). \]
\end{dfn}
\subsection{HyperK\"ahler manifolds}
In this section, we recall basic terminology on hyperK\"ahler manifolds.

\begin{dfn}
A  hyperK\"ahler manifold is a simply connected smooth projective variety $M$ such that $H^0(\Omega_M^2)$ is spanned by an everywhere non-degenerate holomorphic two form.
\end{dfn}
In this paper, we always assume that hyperK\"ahler manifolds are projective. 

Let $M$ be a hyperK\"ahler manifold. The second cohomology group of $M$ has the non-degenerated quadratic form $q_M : H^2(M,\mathbb{Z}) \to \mathbb{Z}$ so called the Beauville-Bogomolov-Fujiki form. Then $H^2(M,\mathbb{Z})$ becomes an even unimodular lattice of signature $(3,b_2(M)-3)$ and $q_{M}$ satisfies the Fujiki relation
\[ \int_M \alpha^{2n} = F_M\cdot q_M(\alpha)^n. \]
Here, we write $\mathrm{dim}M=2n$ and $F_M>0$ is the Fujiki constant of $M$.

\begin{dfn}
Let $D$ be a divisor on $M$. The divisor $D$ is called movable if the intersection $\cap_{m\geq1}\mathrm{Bs}|mD|$ of base loci of $mD$ has at least codimension two in $M$. The divisor $D$ is called positive if $q_M(D)>0$ and $DH>0$ hold for a fixed ample divisor $H$.
\end{dfn}

\begin{dfn}
The ample cone $\mathrm{Amp}(M)$ of $M$ is the real cone generated by ample divisors on $M$. The movable cone $\mathrm{Mov}(M)$ of $M$ is the real cone generated by movable divisors on $M$. The positive cone $\mathrm{Pos}(M)$ is the positive cone of $M$ with respect to the Beauville-Bogomolov-Fujiki form. 
\end{dfn}

\subsection{Bridgeland stability conditions on K3 surfaces}
In this subsection, we recall the notion of Bridgeland stability conditions on derived categories of K3 surfaces. Let $S$ be a K3 surface.

\begin{dfn}
Let $E \in D^b(S)$ be an object in $D^b(S)$. We define the Mukai vector $v(E)$ of $E$ as follow:
\[ v(E):= \mathrm{ch}(E)\sqrt{\mathrm{td}(S)} \in \widetilde{H}^{1,1}(S,\mathbb{Z}).\]
The Mukai vector induces the surjective homomorphism
\[ v: K(S) \to \widetilde{H}^{1,1}(S,\mathbb{Z}).\]
\end{dfn}

\begin{rem}
By Hirzebruch-Riemann-Roch theorem, we have $-\chi(E,F)=\langle v(E),v(F) \rangle$ for any $E,F \in D^b(S)$. So the Mukai vector gives isometry  $v : K_{\mathrm{num}}(S) \to  \widetilde{H}^{1,1}(S,\mathbb{Z})$ with respect to the Euler pairing $-\chi(-,-)$ and the Mukai pairing $\langle-,- \rangle$. 
\end{rem}

\begin{dfn}$($\cite{Bri07}$)$
Fix a norm $||-||$ on  $K_{\mathrm{num}}(S)\otimes{\mathbb{R}}$. A stability condition $\sigma=(Z,\{\mathcal{P}(\phi)\}_{\phi \in \mathbb{R}})$ is a pair of a group homomorphism $Z:K_{\mathrm{num}}(S) \to \mathbb{C}$ (called a central charge) and a collection $\{\mathcal{P}(\phi)\}_{\phi \in \mathbb{R}}$ of full additive subcategories $\mathcal{P(\phi)}$ (called $\sigma$-semistable objects with phase $\phi$) such that the followings hold.
\begin{itemize}
\item For any $\phi \in \mathbb{R}$ and $0 \neq E \in \mathcal{P}(\phi)$, we have $Z(E) \in \mathbb{R}_{>0}e^{i \pi \phi}$.
\item For any $\phi \in \mathbb{R}$, we have $\mathcal{P}(\phi+1)=\mathcal{P}(\phi)[1]$.
\item If $\phi_1 > \phi_2$, then we have $\mathrm{Hom}(E_1,E_2)=0$ for any $E_i \in \mathcal{P}(\phi_i)$.
\item For any $E \in D^b(S)$, there exists exact triangles $E_{i-1} \to E_i \to F_i$ $(1 \leq i \leq n)$ and $\phi_1 > \phi_2 > \cdot \cdot \cdot >\phi_n$ such that $F_i \in \mathcal{P}(\phi_i)$ and $E_0=0, E_n=E$. 
This property is called the Harder-Narasimhan property.
\item There exists a constant $C>0$ such that $||E||<C\cdot |Z(E)|$ for any $0 \neq E \in \cup_{\phi \in \mathbb{R}} \mathcal{P}(\phi)$. This property is called the support property.
\end{itemize}
We denote the set of all stability conditions on $D^b(S)$ by $\mathrm{Stab}(S)$. 
For $\phi \in \mathbb{R}$, a simple object of $\mathcal{P}(\phi)$ is called a $\sigma$-stable object with phase $\phi$.
\end{dfn}

Bridgeland \cite{Bri07} proved that there is a structure of a complex manifold on $\mathrm{Stab}(S)$. 

\begin{thm}$($\cite{Bri07}$)$
There is the topology on  $\mathrm{Stab}(S)$ such that the map 
\[\mathrm{Stab} (S) \to \mathrm{Hom}_{\mathbb{Z}}(K_{\mathrm{num}}(S) ,\mathbb{C}), (Z,\mathcal{P}) \mapsto Z \]
is a local homeomorphism. In particular, the space of stability conditions $\mathrm{Stab}(S)$ on $S$ has a structure of a complex manifold. 
\end{thm}
 
\begin{rem}$($\cite{Bri07}$)$
Let $\sigma=(Z,\mathcal{P}) \in \mathrm{Stab}^*(S)$ be a stability condition. For an interval $I$, we denote by $\mathcal{P}(I)$ the extension-closed subcategory generated by objects $E \in \mathcal{P}(\phi)$ for all $\phi \in I$.  Then $\mathcal{P}((0,1])$ is the heart of a bounded t-structure in $D^b(S)$.
\end{rem}

 We recall the properties of spaces of stability conditions on K3 surfaces. Let $\mathcal{P}(S)$ be the set of all classes in  $\widetilde{H}^{1,1}(S,\mathbb{Z})\otimes \mathbb{C}$ , whose real and imaginary parts span a positive definite real plane in $ \widetilde{H}^{1,1}(S,\mathbb{Z}) \otimes \mathbb{R}$. The subset $\mathcal{P}^+(S)$ is the connected component of $\mathcal{P}(S)$, which contains a class $e^{i\omega} \in \widetilde{H}^{1,1}(S,\mathbb{Z})\otimes \mathbb{C}$ for an ample divisor $\omega$ on $S$. Let $\Delta(S)$ be the set of $(-2)$-classes in $\widetilde{H}^{1,1}(S,\mathbb{Z})$.
We define the subset $\mathcal{P}^+_0(S)$ as 
\[\mathcal{P}^+_0(S):=\mathcal{P}^+(S) \setminus \bigcup_{\delta \in \Delta(S)}\delta^\perp.\]

\begin{thm}[\cite{Bri08}]
There is a connected component $\mathrm{Stab}^*(S) \subset \mathrm{Stab}(S)$  such that the map $\pi : \mathrm{Stab}^*(S) \to \mathcal{P}^+_0(S)$ which sends $(Z,\mathcal{P}) \mapsto \Omega$ is a covering map. Here, the class $\Omega$ is determined by the property $Z(-)=\langle \Omega, - \rangle$. The group of all autoequivalences preserving $\mathrm{Stab}^*(S)$ and acting trivially on $H^*(S,\mathbb{Z})$ is the group of deck transformations of $\pi$.
\end{thm}

\begin{conj}[\cite{Bri08}]\label{Briconj}
The group $\mathrm{Aut}(D^b(S))$ of autoequivalences of $D^b(S)$ preserves the connected component $\mathrm{Stab}^*(S)$. Moreover, $\mathrm{Stab}^*(S)$ is simply connected.
\end{conj}

\begin{thm}$($\cite{BB}$)$
Assume that $S$ has Picard number $\rho(S)=1$. Then Conjecture \ref{Briconj} is true.
\end{thm}

The following proposition is a stronger version of Theorem \ref{torelli}. We will use this version in Section 3.

\begin{prop}$($\cite{Huy14}$)$\label{Huysurj}
Let $\mathrm{Aut}^*(D^b(S))$ be the group of autoequivalences of $D^b(S)$ preserving $\mathrm{Stab}^*(S)$.
The homomorphism 
\[(-)^H :\mathrm{Aut}^*(D^b(S)) \to \OHH(H^*(S,\ZZ))\]
is surjective.
\end{prop} 

We recall wall and chamber structures on spaces of stability conditions on K3 surfaces.
\begin{dfn}
Let $v$ be a primitive Mukai vector. A stability condition $\sigma \in \mathrm{Stab}^*(S)$ is called $v$-generic if there is no strictly $\sigma$-semistable objects with Mukai vector $v$.   
\end{dfn}

\begin{prop}$($\cite{Bri08}, \cite{Tod08}, \cite{BM13})
Fix a primitive Mukai vector $v \in   \widetilde{H}^{1,1}(S,\mathbb{Z})$. There is a locally finite set of walls (real codimension one submanifolds) in $\mathrm{Stab}(S)$, which satisfies the following properties.  We call a connected component of the complement of walls a chamber. 
\begin{itemize}
\item Let $\mathcal{C}$ be a chamber. If $E$ is a $\sigma$-semistable object with Mukai vector $v$, then $E$ is $\tau$-semistable for any $\tau \in \mathcal{C}$. 
\item A stability condition $\sigma \in \mathrm{Stab}^*(S)$ is in some chamber if and only if $\sigma$ is $v$-generic.
\end{itemize}
\end{prop}

In Subsection 2.4, we will recall Bayer and Macri's classification of walls \cite{BM12}, \cite{BM13}.

\subsection{Moduli spaces of Bridgeland stable objects on K3 surfaces}
In this subsection, we recall Bayer and Macri's work on moduli spaces of Bridgeland stable objects on K3 surfaces \cite{BM12}, \cite{BM13}. 
Let $S$ be a K3 surface. Let $v \in \widetilde{H}^{1,1}(S,\mathbb{Z})$ be a primitive Mukai vector with $v^2>0$. Take a $v$-generic stability condition $\sigma \in \mathrm{Stab}^*(S)$. Denote the moduli stack of $\sigma$-semistable objects with Mukai vector $v$ by $\mathcal{M}_\sigma(v)$.  

Bayer and Macri \cite{BM12} proved existence and projectivity of coarse moduli spaces of Bridgeland stable objects on K3 surfaces.
\begin{thm}$($\cite{BM12}, \cite{BM13}$)$\label{moduli}
Let $v \in  \widetilde{H}^{1,1}(S,\mathbb{Z})$ be a primitive Mukai vector with $\langle v,v \rangle \ge -2$. Let $\sigma \in \mathrm{Stab}^*(S)$ be a $v$-generic stability condition.
Then the coarse moduli space $M_{\sigma}(v)$ of $\sigma$-stable objects with Mukai vector $v$ exists as a  hyperK\"ahler manifold of $\text{K3}^{[n]}$-type and $\dim{M_{\sigma}(v)}=2+\langle v,v \rangle$. Note that the coarse moduli space $M_\sigma(v)$ is non-empty.
\end{thm}

Recall the notion of quasi-universal families on the coarse moduli space $M_\sigma(v)$.

\begin{dfn}$($\cite{BM12}$)$
Let $M$ be a connected algebraic space of finite type over $\mathbb{C}$. 
\begin{itemize}
\item An object $\mathcal{E} \in D^b(M \times S)$ is called a quasi-family of objects in $\mathcal{M}_\sigma(v)(\mathrm{Spec}\mathbb{C})$ if for any closed point $p \in M$, there is an object $E \in \mathcal{M}_\sigma(v)$ and a positive integer $\rho >0$ such that  $\mathcal{E}|_{p \times S} \simeq E^{\oplus \rho}$.  Then $\rho$ is independent to a choice of $p \in M$ and called the similitude of $\mathcal{E}$. 
\item Two quasi-families $\mathcal{E}, \mathcal{E}^{\prime} \in D^b(M \times S)$ are called equivalent if there are vector bundles $V, V^{\prime}$ on $M$ such that $\mathcal{E} \otimes p_{M}^*V \simeq \mathcal{E}^{\prime} \otimes p_{M}^*V^\prime$, 
\item  A quasi-family $\mathcal{E} \in D^b(M \times S)$ is called a quasi-universal family if for any algebraic scheme $M^\prime$ and a quasi-family $\mathcal{E}^\prime \in D^b(M^\prime \times S)$, there is a unique morphism $f : M^\prime \to M$ such that $\mathbf{L}f^* \mathcal{E}$ is equivalent to $\mathcal{E}^\prime$.    
\end{itemize}
\end{dfn}

\begin{rem}$($\cite{BM12}$)$
Let $v \in  \widetilde{H}^{1,1}(S,\mathbb{Z})$ be a primitive Mukai vector with $\langle v,v \rangle \ge -2$. Let $\sigma \in \mathrm{Stab}^*(S)$ be a $v$-generic stability condition. Then the coarse moduli space $M_\sigma(v)$ has a quasi-universal family $\mathcal{E} \in D^b(M_\sigma(v) \times S)$,  unique up to equivalence.
\end{rem}

Using it, we obtain the following Hodge isometry, so called Mukai homomorphism.

\begin{prop}$($\cite{BM12}$)$\label{Mukai iso}
Let $v \in  \widetilde{H}^{1,1}(S,\mathbb{Z})$ be a primitive Mukai vector with $\langle v,v \rangle >0$. Let $\sigma \in \mathrm{Stab}^*(S)$ be a $v$-generic stability condition. Let $\mathcal{E} \in D^b(M_\sigma(v) \times S)$ be a quasi-universal family of similitude $\rho$.  We define a homomorphism  $\theta : v^{\perp} \to H^2(M_\sigma(v),\mathbb{Z})$ as  follow:
\[ \theta_{v}(w) := \frac{1}{\rho}\cdot [\Phi_{\mathcal{E}}^H(w)]_2. \]
Here, $[-]_2$ means the degree two part and we consider the orthogonal complement $v^\perp$ in $H^*(S,\mathbb{Z})$.
Then the homomorphism $\theta_v : v^\perp \to H^2(M_\sigma(v),\mathbb{Z})$ is a Hodge isometry and independent to a choice of $\mathcal{E}$. 

\end{prop}

Bayer and Macri \cite{BM12} constructed nef divisors on $M_\sigma(v)$ from stability conditions on $D^b(S)$.
Let $\mathcal{C} \in \mathrm{Stab}^*(S)$ be the chamber containing the stability conditon $\sigma$. Then we write $M_\mathcal{C}(v):=M_\sigma(v)$.
Let $\mathcal{E} \in D^b(M_\sigma(v) \times S)$ be a quasi-universal family of similitude $\rho$.

\begin{thm}$($\cite{BM12}$)$
For a stability condition $\tau=(Z,\mathcal{P}) \in \overline{\mathcal{C}}$, we define a divisor $l_\tau \in \mathrm{NS}(M_\mathcal{C}(v))_{\mathbb{R}}$  as follow:
\[ l_\tau([C]):= \mathrm{Im}\biggl(-\frac{Z(v(\Phi_\mathcal{E}(\mathcal{O}_C))}{Z(v)}\biggr) \in \mathbb{R}. \] 
Here, $[C] \in N_1(M_\mathcal{C}(v))$ is a numerical class of a curve $C$ on $M_\mathcal{C}(v)$.
Then $l_\tau \in \mathrm{NS}(M_\mathcal{C}(v))$ is a well-defined nef divisor. If $\tau$ is $v$-generic, then we have $l_\tau \in \mathrm{Amp}(M_\mathcal{C}(v))$.
\end{thm}
Let $\sigma_0 \in \mathrm{Stab}^*(S)$ be a generic stability condition on a wall $\mathcal{W}$. Let $\sigma_+$ and $\sigma_-$ be $v$-generic stability conditions nearby $\mathcal{W}$ in opposite chambers. 
\begin{thm}$($\cite{BM12}$)$
The divisors $\l_{\sigma_0, \pm} \in \mathrm{NS}(M_{\sigma_\pm}(v))$ are nef and big.  Moreover, they induce birational contractions $\pi^\pm : M_{\sigma_\pm}(v) \to \overline{M}_\pm$ contracting objects which are S-equivalent each other with respect to $\sigma_0$.
\end{thm}

We recall the classification of walls.
\begin{dfn}$($\cite{BM12}, \cite{BM13}$)$\label{walldef}
We call a wall $\mathcal{W}$:
\begin{itemize}
\item a fake wall, if $\pi_\pm  : M_{\sigma_\pm}(v) \to \overline{M}_\pm$ is an isomorphism.  
\item a totally semistable wall, if $M_{\sigma_0}^\mathrm{st}(v)=\emptyset$..
\item a flopping wall, if we can identify $\overline{M}_{+}=\overline{M}_{-}$ and the $\pi^{\pm}$ induce flopping contractions.
\item a divisorial wall, if the morphisms $\pi^{\pm}: M_{\sigma_{\pm}}(v) \to \overline{M}_{\pm}$ are both divisorial contractions
\end{itemize}
\end{dfn}

Bayer and Macri \cite{BM13} classified walls in term of Mukai lattices.
\begin{dfn}$($\cite{BM13}$)$
For a wall $\mathcal{W} \subset \mathrm{Stab}^*(S)$, we define a sublattice $\mathcal{H}_{\mathcal{W}} \subset \widetilde{H}^{1,1}(S,\mathbb{Z})$ as follow:
\[ \mathcal{H}_{\mathcal{W}} := \bigl\{ w \in \widetilde{H}^{1,1}(S,\mathbb{Z}) \mid \mathrm{Im}\frac{Z(w)}{Z(v)}=0 \ \text{for all}\  (Z,\mathcal{P}) \in \mathcal{W} \bigr\}. \] 
Then $\mathcal{H}_{\mathcal{W}}$ is a rank two hyperbolic lattice containing the Mukai vector $v$.
\end{dfn}

\begin{dfn}$($\cite{BM13}$)$
Let $\mathcal{H} \subset \widetilde{H}^{1,1}(S,\mathbb{Z})$ be a hyperbolic sublattce. A class $w \in \mathcal{H}$ is called positive if $w^2 \geq 0$ and  $\langle v,w \rangle >0$ hold. 
\end{dfn}

The following is the classification of walls.

\begin{thm}$($\cite{BM13}$)$\label{wall}
Let $\mathcal{W} \subset \mathrm{Stab}^*(S)$ be a wall. 
\begin{itemize}
\item[(a)] The wall $\mathcal{W}$ is a divisorial wall if one of the three conditions hold:

\bf{(Brill-Noether):} \it there exists a $(-2)$-class $a \in \mathcal{H}_{\mathcal{W}}$ with $\langle a, v \rangle =0$.

\bf{(Hilbert-Chow):} \it there exists an isotropic class $a \in \mathcal{H}_{\mathcal{W}}$ with $\langle a, v \rangle =1$.

\bf{(Li-Gieseker-Uhlenbeck):} \it there exists an isotropic class $a \in \mathcal{H}_{\mathcal{W}}$ with $\langle a, v \rangle =2$.
\item[(b)] Otherwise, if there are positive classes $a, b \in \mathcal{H}_{\mathcal{W}}$  such that $v=a+b$, or if there is a $(-2)$-class $a \in \mathcal{H}_{\mathcal{W}}$ with $0<\langle a, v \rangle \leq v^2/2$, then the wall $\mathcal{W}$ is a flopping wall.  

\item[(c)] Assume that (a) or (b) does not hold. Then the wall $\mathcal{W}$ is a fake wall.
\end{itemize}
\end{thm}

We can deduce the following description of the nef cone $\mathrm{\overline{Amp}}(M_\sigma(v))$ from Theorem \ref{wall}.

\begin{thm}$($\cite{BM13}, \rm{Theorem 12.1}$)$\label{nefcone} \it{}
Consider the chamber decomposition of the positive cone $\overline{\mathrm{Pos}}(M_\sigma(v))$ whose walls are given by linear subspaces of the form
\[ \theta_{v}(v^\perp \cap a^\perp) \]
for all $a \in  \widetilde{H}^{1,1}(S,\mathbb{Z})$ such that $-2 \leq a^2 <v^2/4$ and $0 \leq \langle v, a \rangle \leq v^2/2$.  Then the nef cone $\overline{\mathrm{Amp}}(M_\sigma(v))$ is one of the chambers of above chamber decomposition.
\end{thm}
\begin{rem}
In Theorem \ref{nefcone}, the upper bound $a^2<v^2/4$ comes from the hyperbolicity of the lattice $\mathcal{H}:= \langle v, a \rangle$.
\end{rem}

\begin{thm}$($\cite{BM13}$)$
Let $\mathcal{W} \subset \mathrm{Stab}^*(S)$ be a wall. Let $\sigma_0 \in \mathcal{W}$ be a generic stability condition. Take stability conditions $\sigma_\pm$ nearby $\mathcal{W}$ in opposite chambers.  Then there is a possibly contravariant  autoequivalence $\Phi : D^b(S) \to D^b(S)$ and a common open subset $U \subset M_{\sigma_\pm}(v)$ such that for any $u \in U$, the corresponding objects $E_u \in M_{\sigma_{+}}(v)$ and $F_u \in M_{\sigma_-}(v)$ are related via $F_u = \Phi(E_u)$. If  $\mathcal{W}$ is a fake wall, then we can take an open set $U$ as $M_{\sigma_+}(v), M_{\sigma_-}(v)$ respectively and $\Phi \in \mathrm{Aut}^0(D^b(S))$.
\end{thm}

\subsection{Topological entropy}
In this subsection, we recall the notion of topological entropy. We need only Theorem \ref{GY} and Theorem \ref{firstdeg}.
Let $X$ be a compact topological space with a metric space structure $(X,d)$. Let $f :X \to X$ be a surjective continuous map.  To define the topological entropy, we need the notion of $(n.\epsilon)$-separated subsets of $X$. 
\begin{dfn}$($\cite{Ruf71}$)$
Take a positive real number $\epsilon >0$ and a positive integer $n>0$. Points $x ,y \in X$ are $(n,\epsilon)$-separated if  $\mathrm{max}_{0 \leq i \leq n-1}d(f^i(x),f^i(y)) \geq \epsilon$. A subset $F \subset X$ is called $(n,\epsilon)$-separated if any two distinct points $x,y \in F$ are $(n,\epsilon)$-separated. Due to compactness of $X$, we can prove that 
\[N_d(n,\epsilon):=\mathrm{max}\{\# F \mid \text{$F \subset X$ is $(n,\epsilon)$-separated}.\}\] 
is finite.
\end{dfn}

\begin{dfn}$($\cite{Ruf71}$)$
The topological entropy $h_{\mathrm{top}}(f)$ of $f$ is defined as follow:
\[ h_{\mathrm{top}}(f):= \lim_{\epsilon \to +0}\limsup_{n \to \infty} \frac{\log N_d(n,\epsilon)}{n} \in [0,\infty].\]
The topological entropy $h_{\mathrm{top}}(f)$ of $f$ is independent of a choice of a metric $d$ on the topological space $X$. 
\end{dfn}
Topological entropy is related to spectral radii.
\begin{dfn}
Let $V$ be a finite dimensional vector space and $\phi : V \to V$ be an linear map. The spectral radius $\rho(\phi)$ of $\phi$ is the maximum of the absolute value of eigen values of $\phi$. 
\end{dfn}
We can compute topological entropy by the following theorems. 
\begin{thm}$($\cite{Gro1}, \cite{Gro2}, \cite{Yom}$)$ \label{GY}
Let $X$ be a smooth projective variety. Consider a surjective holomorphic endomorphism $f : X \to X$.
Then the following holds:
\begin{align*}
h_{\mathrm{top}}(f)&=\log \rho(f^*|_{\bigoplus_{p=0}^{\mathrm{dim}X} H^{p,p}(X,\mathbb{Z})})\\
&=\log \rho(f^*|_{\bigoplus_{p=0}^{\mathrm{dim}X} H^{2p}(X,\mathbb{Z})}).                   
\end{align*}
\end{thm}

\begin{thm}\label{firstdeg}$($\cite{Ogu0}$)$
Let $M$ be a  hyperK\"ahler manifold. Let $f:M \to M$ be a surjective holomorphic endomorphism. Then we have 
\[ h_{\mathrm{top}}(f)=\frac{1}{2}\mathrm{dim}M \cdot \log \rho(f^*|_{H^2(M,\mathbb{Z})}).\]
\end{thm}

\subsection{Categorical entropy}
In this subsection, we recall the notion of the categorical entropy of endofunctors on triangulated categories \cite{DHKK}. Let $\mathcal{D}$ be a triangulated category. Consider an endofunctor $\Phi : \mathcal{D} \to \mathcal{D}$. 

\begin{dfn}
An object $G$ in $\mathcal{D}$ is called a splitting generator if for any object $E \in \mathcal{D}$, there are exact triangles $E_{i-1} \to E_i \to G[n_i]$ $(1 \leq i \leq k)$ and some object $E^{\prime} \in \mathcal{D}$ such that $E_0=0, E_k=E \oplus E^{\prime}$. 
\end{dfn}

\begin{thm}$($\cite{Orl09}$)$
Let $X$ be a quasi-projective scheme. Let $\mathcal{O}_X(1)$ be a very ample line bundle on $X$.
Then $G:=\bigoplus_{i=0}^{\mathrm{dim}X}\mathcal{O}_X(i)$ is a splitting generator of the triangulated category  $\mathrm{Perf}(X)$ of perfect comlexes on $X$. In particular, if $X$ is smooth, $G$ is a splitting generator of $D^b(X)$.
\end{thm}
\begin{dfn}$($\cite{DHKK}$)$
Let $E$ and $F$ be objects in $\mathcal{D}$. Take a real number $t \in \mathbb{R}$.
We define the complexity $\delta_{t}(G,E)$ of $F$ with respect to $E$ as follow:
\[ \delta_{t}(G,E):=\inf \biggl\{\sum_{i=1}^{k}e^{n_i t} \mid E_{i-1} \to E_i \to G[n_i] (1 \leq i \leq k) ,E_0=0, E_k=E \oplus E^{\prime}\biggr\}. \] 
\end{dfn}
Using complexity, we can define the notion of the categorical entropy. 
\begin{dfn}$($\cite{DHKK}$)$
Let $G$ be a splitting generator of $\mathcal{D}$. We define the categorical entropy $h_{t}(\Phi)$ of $\Phi$ at $t \in \mathbb{R}$ as follow:
\[ h_{t}(\Phi):= \lim_{n \to \infty} \frac{\log \delta_{t}(G, \Phi^n(G))}{n} \in \mathbb{R} \cup \{-\infty\}. \]
We call the entropy $h_{0}(\Phi)$ of $\Phi$ at $0$ the categorical entropy $h_{\mathrm{cat}}(\Phi)$ of $\Phi$. 
Then we have $h_{\mathrm{cat}}(\Phi) \geq 0$. The entropy of $\Phi$ is independent to a choice of a generator $G$.
\end{dfn}

The following is the analogue of Theorem \ref{GY}.
\begin{thm}$($\cite{DHKK}$)$
Let $X$ be a smooth projective variety with $H^{\mathrm{odd}}(X, \mathbb{C})=0$. Consider an autoequivalence $\Phi \in \mathrm{Aut}(D^b(X))$ and let $\Phi^H : H^*(X, \mathbb{C}) \to H^*(X,\mathbb{C})$ be the induced linear map on the cohomology group. 
Then $h_{\mathrm{cat}}(\Phi) \geq \log \rho (\Phi^H)$.
\end{thm}
\begin{proof}
This is Theorem 2.8 in \cite{DHKK}. The assumption  $H^{\mathrm{odd}}(X, \mathbb{C})=0$ implies the assumption in Lemma 2.7 in \cite{DHKK}.
\end{proof}
For a K3 surface $S$, we have $H^{\mathrm{odd}}(S,\mathbb{C})=0$.
Kikuta and Takahashi \cite{KT16} proposed the following conjecture.

\begin{conj}$($\cite{KT16}$)$\label{KTconj}
Let $X$ be a smooth projective variety. For $\Phi \in \Aut(D^b(X))$, we have 
\[ \hcat(\Phi)=\log \rho([\Phi]).\]
Here, $[\Phi]$ is the induced isomorphism on $K_{\mathrm{num}}(X)$.
\end{conj}

\subsection{Automorphisms of hyperK\"ahler manifolds of Picard number two}
In this subsection, we collect important facts about automorphisms of hyperK\"ahler manifolds of Picard number two.
First, Hilbert schemes of points on K3 surfaces of Picard number one have finite birational automorphism groups. 
\begin{prop}\label{hilb}$($\cite{Ogu1}$)$
Let $S$ be a K3 surface of Picard number one. Let $n>0$ be a positive integer.
Then the biratiuonal automorphism group $\mathrm{Bir}(\mathrm{Hilb}^n(S))$ of the Hilbert scheme of points on $S$ is a finite group.
\end{prop}

The following theorem due to Oguiso \cite{Ogu1} is important in the proof of the main theorem.
\begin{thm}\label{Oguiso}$($\cite{Ogu1}$)$
Let $M$ be a hyperK\"ahler manifold of Picard number two. Let $l_1 =\mathbb{R}_{\geq0}x_1$ and $\l_2=\mathbb{R}_{\geq0}x_2$ be boundary rays of the nef cone $\overline{\mathrm{Amp}}(M)$. 
Then the followings hold.
\begin{enumerate}
\item The boundary ray $l_1$ is rational if and only if the boundary ray $l_2$ is rational.
\item If the boundary ray $l_1$ is rational, then the automorphism group $\mathrm{Aut}M$ is a finite group.
\item If the boundary ray $l_1$ is irrational, then $\overline{\mathrm{Amp}}(M)=\overline{\mathrm{Mov}}(M)=\overline{\mathrm{Pos}}(M)$ and $\mathrm{Aut}M=\mathrm{Bir}M$ is an infinite group. Moreover, an automorphism $f \in \mathrm{Aut}M$ with $\mathrm{ord}(f)=\infty$ has the positive topological entropy.
\end{enumerate}
\end{thm}
The following theorem due to Amerik and Vervitsky \cite{AV16} ensures the existence of hyperK\"ahler manifolds satisfying Theorem \ref{Oguiso} (3).

\begin{thm}$($\cite{AV16}$)$
Let $M$ be a (possibly non-projective) hyperK\"ahler manifold with $b_2(M)\geq 5$. Then $M$ admits a projective deformation $M^{\prime}$ with infinite group of symplectic automorphisms and Picard number two.
\end{thm}
\section{MAIN THEOREMS} 
\subsection{Derived automorphisms of positive entropy on K3 surfaces}
In this subsection, we will prove the existence of autoequivalences of derived categories of K3 surfaces with positive entropy.

\begin{thm}\label{posent}
Let $S$ be a K3 surface. Then there is an autoequivalence $\Phi \in \mathrm{Aut}(D^b(S))$ such that $h_{\mathrm{cat}}(\Phi) \geq  \log{\rho(\Phi^H)} >0$. 
\end{thm}
\begin{proof}
We will construct $\Phi \in \mathrm{Aut}(D^b(S))$ as a composition of two spherical twists.
The first spherical twist is defined by $\Phi_1 := \mathrm{ST}_{\mathcal{O}_S}$.
We will define the second spherical twist $\Phi_2$. Take an ample divisor $h$ on $S$ and write $h^2=2d$. Consider a $(-2)$-class \[w:=(d+1,dh,d^2-d+1) \in  \widetilde{H}^{1,1}(S,\mathbb{Z}).\]  
Take a general ample divisor $H$ on $S$ with respect to $w$.  Then there is a $H$-stable sheaf $E$ with Mukai vector $v(E)=w$ by non-emptyness result (\cite{Yos99}, Theorem 0.1). 
Since $w^2=-2$,  $E$ is a spherical object. We define the second spherical twist by $\Phi_2:=\mathrm{ST}_{E}$. We set $\Phi := \Phi_1 \circ \Phi_2 \in  \mathrm{Aut}(D^b(S))$.
Consider a sublattice $L$ of the Mukai lattice defined by 
\[ L:=H^0(S,\mathbb{Z}) \oplus \mathbb{Z}h \oplus H^4(S,\mathbb{Z}). \]
The representation matrices of $\Phi_1^H|_L,\Phi_2^H|_L$ with respect to $(1,0,0), (0,h,0),(0,0,1) \in L$ are as follow:
\[\Phi_1^H|_L=\begin{pmatrix}0 & 0 & -1\cr 0 & 1 & 0\cr -1 & 0 & 0\end{pmatrix}, \Phi_2^H|_L=\begin{pmatrix}-d^3 & 2d^2(d+1) & -(d+1)^2\cr -d(d^2-d+1) & 2d^3+1 & -d(d+1)\cr -(d^2-d+1) & 2d^2(d^2-d+1) & -d^3\end{pmatrix}.\]
So we have
\[ \Phi^H|_L=\begin{pmatrix}(d^2-d+1)^2 & -2d^2(d^2-d+1) & d^3\cr -d(d^2-d+1) & 2d^3+1 & -d(d+1)\cr d^3& -2d(d+1) & (d+1)^2\end{pmatrix}.\]
The eigen values of  $\Phi^H|_L$ are 
\[1, \frac{d^4+4d^2 \pm (d^3+2d) \sqrt{d^2+4}}{2}.\]
Since $d \geq 1$, we have
\[ h_\mathrm{cat}(\Phi) \geq \log{\rho(\Phi^H)} \geq \log{\frac{7+3\sqrt{5}}{2}}>0.\]
\end{proof}
\begin{rem}
Let $S$ be a K3 surface with $\mathrm{NS}(S)=\mathbb{Z}h$ and $h^2=2d$. If $d=1$, then $\mathrm{Aut}(S)=\mathbb{Z}/2\mathbb{Z}$. If $d>1$, then $\mathrm{Aut}(S)=1$. However, there is an autoequivalence $\Phi \in \mathrm{Aut}(D^b(S))$ with $h_{\mathrm{cat}}(\Phi) \geq \log \rho(\Phi^H)>0$ by Theorem \ref{posent}. 
\end{rem}

In the next subsection, we will discuss the relation between autoequivalences on K3 surfaces and automorphisms of moduli spaces of stable objects. Unfortunately, there is an autoequivalence on a K3 surface with positive entropy such that it does not induce automorphisms of moduli spaces of stable objects.

\begin{ex}\label{counter}
Let $S$ be a K3 surface with $\mathrm{NS}(S)=\mathbb{Z}h$ and $h^2=4$. Consider a Mukai vector $v:=(1,0,-1)$. Then there is an autoequivalence $\Phi \in \mathrm{Aut}(D^b(S))$ such that $h_{\mathrm{cat}}(\Phi) \geq \log \rho(\Phi^H) >0$ and $\Phi^H(v)=v$.  Take a $v$-generic stability condition $\sigma \in \mathrm{Stab}^*(S)$.  Then the moduli space $M_\sigma(v)$ is birational to the Hilbert scheme $\mathrm{Hilb}^2(S)$ of two points on $S$. So $\mathrm{Bir}(M_\sigma(v))$ is finite by Proposition \ref{hilb}. In particular, $\Phi$ does not induce an automorphism on $M_\sigma(v)$.
\end{ex}
\begin{proof}
Consider a $(-2)$-class $w:=(3,2h,3) \in \widetilde{H}^{1,1}(S,\mathbb{Z})$. This is used in the proof of Theorem \ref{posent}.  
So we can take a spherical object $E \in D^b(S)$ with Mukai vector $v(E)=w$. By Theorem \ref{posent}, $\Phi:=\mathrm{ST}_{\mathcal{O}_S} \circ \mathrm{ST}_E$ satisfies $h_{\mathrm{cat}}(\Phi) \geq \log \rho(\Phi^H) >0$.
Since $v=(1,0,-1)$ is orthogonal to $v(\mathcal{O}_S)=(1,0,1)$ and $w$, we have $\Phi^H(v)=v$. 
\end{proof}
\subsection{Crossing fake walls}
In this subsection, we discuss the relation between autoequivalences of K3 surfaces and automorphisms of moduli spaces of stable objects. Let $S$ be a K3 surface.  First, we prove that automorphisms of moduli space of stable objects with infinite order induce autoequivalences on $S$.

\begin{prop}\label{infiniteorder}
Let $v \in \widetilde{H}^{1,1}(S,\mathbb{Z})$ be a primitive Mukai vector with $v^2>0$.  Let $\sigma \in \mathrm{Stab}^*(S)$ be a $v$-generic stability condition. If $\phi \in \mathrm{Aut}(M_\sigma(v))$ is an automorphism of infinite order, then there is a positive integer $n>0$ and $\Phi \in \mathrm{Aut}^*(D^b(S))$ such that  $\Phi^H(v)=v$ and the following diagram commutes.

\[\xymatrix{v^\perp \ar[r]^{\Phi^H} \ar[d]_{\theta_{\sigma,v}} & v^\perp \ar[d]_{\theta_{\sigma,v}} \\ H^2(M_\sigma(v),\mathbb{Z}) \ar[r]^{\phi^{*n}} & H^2(M_\sigma(v),\mathbb{Z})} \]

 Here, we take the orthogonal $v^{\perp}$ in $H^*(S,\mathbb{Z})$ and all homomorphisms are Hodge isometries.
\end{prop}
\begin{proof}
Consider the composition $g:=\theta_{\sigma, v}^{-1} \circ \phi^* \circ \theta_{\sigma, v}$. By lemma \ref{kernel}, we have $\mathrm{ord}(g)=\infty$. Since $d(v^{\perp})$ is a finite group, there is a positive integer $m>0$ such that $\overline{g^m}=1$. So we can take a Hodge isometry  $\psi \in \OH(H^*(S,\ZZ))$ such that $\psi|_{v^{\perp}}=g^{m}, \psi(v)=v$ by Lemma \ref{Nik}.  Then $\psi^2$ preserves positive definite $4$-spaces in $H^*(S,\mathbb{R})$. By Proposition \ref{Huysurj}, there is an autoequivalence $\Phi \in  \mathrm{Aut}^*(D^b(S))$ such that $\Phi^H=\psi^2$. Putting $n:=2m$, we have the commutative diagram in the statement.
\end{proof}
By Example \ref{counter}, an autoequivalence of infinite order does not necessarily induce an automorphism of a moduli space of stable objects. We will consider the converse problem in Theorem \ref{converse} . First, we show the following proposition.

\begin{prop}\label{commute}
Let $v \in \widetilde{H}^{1,1}(S,\mathbb{Z})$ be a primitive Mukai vector with $v^2>0$.  Consider an autoequivalence $\Phi \in \mathrm{Aut}(D^b(S))$.  Let $\sigma \in \mathrm{Stab}^*(S)$ be a $v$-generic stability condition and $\tau \in \mathrm{Stab}^*(S)$ be a $\Phi^H(v)$-generic stability condition. If $\Phi$ induces an isomorphism $\Phi : M_\sigma(v) \to M_\tau(\Phi^H(v))$ defined by $[E] \mapsto [\Phi(E)]$, then we have the following commutative diagram:
\[\xymatrix{v^{\perp} \ar[d]_{\theta_{\sigma,v}} \ar[r]^{\Phi^H} & v^\perp  \ar[d]_{\theta_{\tau, v}} \\
H^2(M_\sigma(v),\mathbb{Z}) \ar[r]^{\Phi_*} & H^2(M_{\tau}(v),\mathbb{Z})}  \]
\end{prop}
The following theorem is deduced from Proposition \ref{commute}.
\begin{thm}\label{converse}
Let $v \in \widetilde{H}^{1,1}(S,\mathbb{Z})$ be a primitive Mukai vector with $v^2>0$. Let $\sigma \in \mathrm{Stab}^*(S)$ be a $v$-generic stability condition. Consider an autoequivalence $\Phi \in \mathrm{Aut}(D^b(S))$ such that $\Phi_*\sigma \in \mathrm{Stab}^*(S)$ and $\Phi^H(v)=v$.  If there is a generic path $\gamma : [0,1] \to \mathrm{Stab}^*(S)$ crossing only fake walls with $\gamma(0)=\Phi_*\sigma$ and $\gamma(1)=\sigma$, then there is $\Phi_\gamma \in \mathrm{Aut}^0(D^b(S))$ such that $\Phi_\gamma \circ \Phi$ induces an automorphism $\Phi_\gamma \circ \Phi : M_\sigma(v) \to M_\sigma(v)$ satisfying the following commutative diagram:
\[\xymatrix{v^{\perp} \ar[d]_{\theta_{\sigma,v}} \ar[r]^{\Phi^H} & v^\perp \ar[r]^{\Phi_\gamma^H=1} \ar[d]_{\theta_{\Phi_*\sigma, v}} &v^\perp \ar[d]_{\theta_{\sigma,v}} \\
H^2(M_\sigma(v),\mathbb{Z}) \ar[r]^{\Phi_*} & H^2(M_{\Phi_*\sigma}(v),\mathbb{Z}) \ar[r]^{\Phi_{\gamma*}} & H^2(M_\sigma(v),\mathbb{Z})}\]
Here, all homomorphisms are Hodge isometries and we take the orthogonal complement $v^\perp$ in $H^*(S,\mathbb{Z})$.
\end{thm}
\begin{proof}[Proof of Proposition \ref{commute}]
Let $\mathcal{E} \in D^b(M_\sigma(v) \times S)$ be a quasi-universal family of similitude $\rho$. Let $\mathcal{F} \in D^b(M_\tau(\Phi^H(v)) \times S)$ be a Fourier-Mukai kernel of the composition $\Phi\Phi_\mathcal{E}\phi^* : D^b(M_\tau(\Phi^H(v))) \to D^b(S)$ of Fourier-Mukai functors. Here, we denote the isomorphism $\Phi  : M_\sigma(v) \to M_\tau(\Phi^H(v))$ by $\phi : M_\sigma(v) \to M_\tau(\Phi^H(v))$. It is enough to show that $\mathcal{F}$ is a quasi-universal family of $M_\tau(\Phi^H(v))$ by Proposition \ref{Mukai iso}. We can prove that $\mathcal{F}$ is a quasi-family. In fact, for any point $[E] \in M_\tau(\Phi^H(v))$, we have $\Phi_{\mathcal{F}}(\mathcal{[E]}) \simeq E^{\oplus \rho}$  by the definition of $\mathcal{F} \in D^b(M_\tau(\Phi^H(v)) \times S)$.
Take a quasi-universal family $\mathcal{E}^{\prime} \in D^b(M_\tau(\Phi^H(v)) \times S)$ of similitude  $\rho^{\prime}$. Then there exists a unique morphism $f : M_\tau(\Phi^H(v)) \to M_\tau(\Phi^H(v))$ such that $\mathbf{L}(f \times \mathrm{id}_S)^* \mathcal{E}^{\prime}$ and $\mathcal{F}$ are equivalent. So there are vector bundles $V_1$ and $V_2$ on $M_\tau(\Phi^H(v))$ such that $\mathbf{L}(f \times \mathrm{id}_S)^* \mathcal{E}^{\prime} \otimes p^*V_1 \simeq \mathcal{F} \otimes p^*V_2$, where $p :M_\tau(\Phi^H(v)) \times S \to M_\tau(\Phi^H(v))$ is the projection. Let $[E] \in M_\tau(\Phi^H(v))$ be a point. We have the isomorphisms
\begin{align*}
E^{\rho \cdot \mathrm{rk}V_2}& \simeq (\mathcal{F} \otimes p^*V_2)|_{[E] \times S} \\
                     & \simeq (\mathbf{L}(f \times \mathrm{id}_S)^* \mathcal{E}^{\prime} \otimes p^*V_1)|_{[E] \times S}\\
                     &\simeq \mathbf{L}(f \times \mathrm{id}_S)^* \mathcal{E}^{\prime}|_{[E]\times S}^{\oplus \mathrm{rkV_1}} \\
                     &\simeq (\mathbf{L}(f \times \mathrm{id}_S)^* \mathcal{E}^{\prime} \otimes p^*f^*V_1)|_{[E] \times S} \\
                     &\simeq (\mathbf{L}(f \times \mathrm{id}_S)^* \mathcal{E}^{\prime} \otimes \mathbf{L}(f \times \mathrm{id}_S)^*p^*V_1)|_{[E] \times S} \\
                     &\simeq \mathbf{L}(f \times \mathrm{id}_S)^*(\mathcal{E}^{\prime} \otimes p^*V_1)|_{[E] \times S} \\
                     &\simeq f([E])^{\oplus \rho^{\prime} \cdot \mathrm{rk}V_1}.                   
\end{align*}
The last isomorphism is due to the following commutative diagram:
\[\xymatrix{[E] \times S \ar[d] \ar[r]^{f \times \mathrm{id}_S} & f([E]) \times S \ar[d] \\ M_\tau(\Phi^H(v)) \times S \ar[r]^{f \times \mathrm{id}_S} & M_\tau(\Phi^H(v)) \times S } \]
Since $E$ and $f(E)$ are $\tau$-stable in the same phase, we have $f(E) \simeq E$ and $\rho \cdot \mathrm{rk}V_1=\rho \cdot \mathrm{rk}V_2$. So we have $f=\mathrm{id}_S$. Hence, we have proved that $\mathcal{F}$ is a quasi-universal family of $M_\tau(\Phi^H(v))$.
\end{proof}

\begin{cor}\label{maincor}
Let $S$ be a K3 surface of Picard number one. Let $v \in  \widetilde{H}^{1,1}(S,\mathbb{Z})$ be a primitive Mukai vector with $v^2>0$. If there does not exist a vector $a  \in  \widetilde{H}^{1,1}(S,\mathbb{Z})$ such that  $-2 \leq a^2 <v^2/4$ and $0 \leq \langle v, a \rangle \leq v^2/2$, then there is an autoequivalence $\Phi \in \mathrm{Aut}(D^b(S))$ such that $\Phi^H(v)=v$ and $\log\rho(\Phi^H)>0$.  In particular, for any $v$-generic stability condition $\sigma \in \mathrm{Stab}^*(S)$, there exist an autoequivalence $\Psi \in \mathrm{Aut}^0(D^b(S))$ such that $\phi:=\Psi\Phi : M_\sigma(v) \to M_\sigma(v)$ is an automorphism with 
\[ \frac{1}{2}\mathrm{dim}M_\sigma(v) \cdot h_{\mathrm{cat}}(\Phi) \geq h_{\mathrm{top}}(\phi)>0. \]
\end{cor}

\begin{proof}
Take a $v$-generic stability condition $\sigma \in \mathrm{Stab}^*(S)$. By Theorem \ref{nefcone} and the assumption for the vector $v$, we obtain $\overline{\mathrm{Amp}}(M_\sigma(v))=\overline{\mathrm{Pos}}(M_\sigma(v))$. By Proposition \ref{Mukai iso} and the assumption for the vector $v$, the boundary rays of the nef cone $\overline{\mathrm{Amp}}(M_\sigma(v))$ are irrational. So there is an automorphism $\psi \in \mathrm{Aut}(M_\sigma(v))$ with $h_{\mathrm{top}}(\psi)>0$ due to Theorem \ref{Oguiso}. 
Using Proposition \ref{infiniteorder}, there exist an positive integer $n$ and an autoequivalence $\Phi \in \mathrm{Aut}^*(D^b(S))$ such that $\Phi^H(v)=v, \Phi^H|_{v^\perp}=\psi^{*n}$ and $\log\rho(\Phi^H)=\log\rho(\psi^{*n}|_{H^2(M_\sigma(v),\mathbb{Z})})>0$. The positivity is deduced from Theorem \ref{firstdeg}. Due to the assumption for the vector $v$, all walls with respect to $v$ are fake walls by Theorem \ref{wall}. So we can show the remaining statement by Theorem \ref{converse}.
\end{proof}

In the next subsection, we will construct examples of Corollary \ref{maincor}.
\subsection{Examples}
In this subsection, we give examples of K3 surfaces and Mukai vectors satisfying the assumption in Corollary \ref{maincor}. First, we see six and eight dimensional examples. We will see a four dimensional example in the next section.

\begin{ex}\label{6dimex}
Let $S$ be a K3 surface with $\mathrm{NS}(S)=\mathbb{Z}h$ and $h^2=132$. Let $v:=(4,h,16) \in \widetilde{H}^{1,1}(S,\mathbb{Z})$ be a Mukai vector. Then $v$ is primitive and $v^2=4$.  Moreover, there does not exist a vector $a  \in  \widetilde{H}^{1,1}(S,\mathbb{Z})$ such that  $-2 \leq a^2 <v^2/4$ and $0 \leq \langle v, a \rangle \leq v^2/2$. 
\end{ex}
\begin{proof}
Let $a=(s,th,u) \in \widetilde{H}^{1,1}(S,\mathbb{Z})$. Since $\langle v,a \rangle=132t-16s-4u \in 4\mathbb{Z}$ and the Mukai lattice is even,  it is enough to consider the equations
\[ \langle v,a \rangle=0, a^2=-2,0 \]
by Corollary \ref{maincor}.
 Then we have $u=33t-4s$ and 
 \[v^\perp= \langle (1,0,-4),(0,h,33) \rangle_{\mathbb{Z}}. \]
 Here, we take the orthogonal complement in  $\widetilde{H}^{1,1}(S,\mathbb{Z})$.
Note that the Gram matrix of $v^\perp$ is $\begin{pmatrix}
                     8 &-33\\ -33 & 132
                     \end{pmatrix}$.                   Since the determinant of the Gram matrix is $33$, $v^\perp$ does not have isotropics. So we consider the equation 
\begin{equation}\label{inteq}
4x^2-33xy+66y^2=-1.
\end{equation}
Taking modulo $3$, we obtain the contradiction $x^2 \equiv 2$ $(\mathrm{mod}3)$. Hence, the equation (\ref{inteq}) does not have integral solutions. So there does not exist a vector $a  \in  \widetilde{H}^{1,1}(S,\mathbb{Z})$ such that  $-2 \leq a^2 <v^2/4$ and $0 \leq \langle v, a \rangle \leq v^2/2$. 
\end{proof}

\begin{ex}\label{8dimex}
Let $S$ be a K3 surface with $\mathrm{NS}(S)=\mathbb{Z}h$ and $h^2=510$.  Let $v:=(6,h,42) \in \widetilde{H}^{1,1}(S,\mathbb{Z})$ be a Mukai vector. Then $v$ is primitive and $v^2=6$.  Moreover, there does not exist a vector $a  \in  \widetilde{H}^{1,1}(S,\mathbb{Z})$ such that  $-2 \leq a^2 <v^2/4$ and $0 \leq \langle v, a \rangle \leq v^2/2$. 
\end{ex}
\begin{proof}
In the same way as the proof of Example \ref{6dimex}, it is sufficient to prove thet the equation 
\begin{equation}\label{inteq2}
7x^2-85xy+255y^2=-1
\end{equation}
does not have integral solutions. Taking mudulo $5$, we obtain the contradiction $x^2 \equiv 2$ $(\mathrm{mod}5)$. Hence the equation (\ref{inteq2}) does not have integral solutions.
\end{proof}
We can construct more examples similarly.

\begin{ex}\label{10dimex}
Let $S$ be a K3 surface with $\mathrm{NS}(S)=\mathbb{Z}h$ and $h^2=1160$.  Let $v:=(8,h,72) \in \widetilde{H}^{1,1}(S,\mathbb{Z})$ be a Mukai vector. Then $v$ is primitive and $v^2=8$.  Moreover, there does not exist a vector $a  \in  \widetilde{H}^{1,1}(S,\mathbb{Z})$ such that  $-2 \leq a^2 <v^2/4$ and $0 \leq \langle v, a \rangle \leq v^2/2$. 
\end{ex}
\begin{proof}
In the same way as the proof of Example \ref{6dimex} and Example \ref{8dimex}, it is sufficient to prove that the equation 
\begin{equation}\label{inteq3}
9x^2-145xy+580y^2=-1
\end{equation}
does not have integral solutions. We can check it using \cite{Al}.
\end{proof}

\begin{ex}\label{12dimex}
Let $S$ be a K3 surface with $\mathrm{NS}(S)=\mathbb{Z}h$ and $h^2=2210$.  Let $v:=(10,h,110) \in \widetilde{H}^{1,1}(S,\mathbb{Z})$ be a Mukai vector. Then $v$ is primitive and $v^2=10$.  Moreover, there does not exist a vector $a  \in  \widetilde{H}^{1,1}(S,\mathbb{Z})$ such that  $-2 \leq a^2 <v^2/4$ and $0 \leq \langle v, a \rangle \leq v^2/2$. 
\end{ex}
\begin{proof}
In the same way as the proof of Example \ref{6dimex}, Example \ref{8dimex} and Example \ref{10dimex}, it is sufficient to prove that the equation 
\begin{equation}\label{inteq4}
9x^2-145xy+580y^2=\pm1
\end{equation}
does not have integral solutions. We can check them using \cite{Al}.
\end{proof}
\section{FROM CUBIC FOURFOLDS}
Let $X$  be a cubic fourfold. Consider the semiorthogonal decomposition
 \[ D^b(X)=\langle \A_X, \mathcal{O}_X, \mathcal{O}_X(1), \mathcal{O}_X(2) \rangle \]
 and the projection functor $\pr : D^b(X) \to \A_X$.
For $i \in \ZZ$, we define $\lambda_i := [\pr(\mathcal{O}_{\mathrm{line}}(i))] \in K_{\mathrm{num}}(\A_X)$. 
Addington proved the following lemmas.

\begin{lem}\label{NS}$($\cite{Ad}, \rm{Proposition 7}$)$\it{}
Let $X$ be a cubic fourfold. 
Then there is a Hodge isometry
\[ \mathrm{NS}(F(X))(-1) \simeq \lambda_1^{\perp}. \]
Here, we take the orthogonal complement  $\lambda_1^{\perp}$ in $K_{\mathrm{num}}(\A_X)$.
\end{lem}

\begin{lem}\label{cubicmukai}$($\cite{Ad}, \rm{Lemma 9}$)$\it{}
If $X \in \C_d$, then there is an element $\tau \in \KA$ such that $\langle \lambda_1, \lambda_2, \tau \rangle$ is a primitive sublattice of discriminant $d$ with Euler pairing
\[\begin{pmatrix}
                     -2 &1 &0 \\ 1 & -2 & 0\\ 0&0 & 2k
                     \end{pmatrix} \ \ \  \text{where} \  d=6k\]
 \[\begin{pmatrix}
                     -2 &1 &0 \\ 1 & -2 & 1\\ 0&1 & 2k
                     \end{pmatrix} \ \ \  \text{where} \  d=6k+2.\]
\end{lem}
For a special cubic fourfold $X$ with $\rk\Hodge=2$, the Picard number of $F(X)$ is two.  
\begin{thm}
Let $X$ be a special cubic fourfold of discriminant $74$. Assume that $\rk\Hodge=2$. Then there is an automorphism $\phi \in \Aut(F(X))$ such that $\htop(\phi)>0$.
\end{thm}
\begin{proof}
Note that $74=6\cdot12+2$. Moerover, $F(X)$ is isomorphic to a moduli space of stable objects in the derived category of some K3 surface by Theorem \ref{Adthm} and \cite{BM13}. By lemma \ref{cubicmukai}, there is $\tau \in K_{\mathrm{num}}(\A_X)$ such that  
 \[K_{\mathrm{num}}(\A_X)=\langle \lambda_1, \lambda_2, \tau \rangle\]
 with the Gram matrix $\begin{pmatrix}
                     -2 &1 &0 \\ 1 & -2 & 1\\ 0&1 & 24
                     \end{pmatrix}$.
By Lemma \ref{NS}, we obtain
\[ \mathrm{NS}(F(X))=\langle \lambda_1+2\lambda_2, \tau \rangle\] 
with the Gram matrix 
$\begin{pmatrix}
                     6 &-2 \\ -2 & -24
                     \end{pmatrix}$.
Since    \[ \mathrm{det}\begin{pmatrix}
                     6 &-2 \\ -2 & -24
                     \end{pmatrix}=-148, \]
                     $\mathrm{NS}(F(X))$ does not contain any isotropies. So it is enough to show that the equation 
                     \[ 6x^2-4xy-24y^2=-2 \]
                     does not have any integral solutions by Corollary \ref{maincor}. 
In fact, we can check it using \cite{Al}.                     
\end{proof}

\address{Graduate School of Mathematical Sciences, University of Tokyo, Meguro-ku, Tokyo 153-8914, Japan}

\it{E-mail-address}\rm{:} genki.oouchi@ipmu.jp
\end{document}